\documentclass[12pt]{amsart}
\usepackage{amssymb, amsmath}
\newtheorem{question}{Question}[section]
\newtheorem{theorem}[question]{Theorem}
\newtheorem{lemma}[question]{Lemma}
\newtheorem{corollary}[question]{Corollary}
\newtheorem{example}[question]{Example}

\newtheorem{definition}[question]{Definition}
\newtheorem{proposition}[question]{Proposition}

\newtheorem{remark}[question]{Remark}
\title{A new class of spaces with all finite powers Lindel\"of}

\author{Natasha May}
\address{Department of Mathematics, Faculty of Science and Engineering, York University, Toronto, ON,  M3J 1P3 Canada}
\curraddr{}
\email{mayn@mathstat.yorku.ca, natashammay@gmail.com}
\thanks{Corresponding author: Natasha May, Tel. +14166502322}

\author{Santi Spadaro}
\address{Department of Mathematics, Faculty of Science and Engineering, York University, Toronto, ON,  M3J 1P3 Canada}
\curraddr{Institute of Mathematics, Silesian University in Opava, Na Rybn\' i\v cku 626/1, 746 01 Opava, Czech Republic}
\email{santispadaro@yahoo.com}

\thanks{The second author was supported by an INdAM-Cofund outgoing fellowship. He is also grateful to the Fields Institute of the University of Toronto for hospitality.}

\author{Paul Szeptycki}
\address{Department of Mathematics, Faculty of Science and Engineering, York University, Toronto, ON,  M3J 1P3 Canada}
\email{szeptyck@yorku.ca}
\thanks{The third author acknowledges support from NSERC grant 238944}

\theoremstyle{remark}
\newtheorem*{notation}{Notation}
\newtheorem*{claim}{Claim}
\newtheorem*{note}{Note}



\begin{document}

\begin{abstract}
We consider a new class of open covers and classes of spaces defined from them, called $\iota $-spaces (``iota spaces").  We explore their relationship with $\epsilon $-spaces (that is, spaces having all finite powers Lindel\"of) and countable network weight.  An example of a hereditarily $\epsilon $-space whose square is not hereditarily Lindel\" of is provided.
\end{abstract}

\maketitle
\section{Introduction}

A topological space in which each finite power is Lindel\" of is called an \emph{$\epsilon $-space}. Equivalently, $X$ is an $\epsilon $-space if every open $\omega $-cover of $X$ has a countable $\omega $-subcover, where a cover of a space $X$ is an \emph{$\omega $-cover} if each finite subset of $X$ is contained in an element of the cover.  A natural generalization of an $\omega $-cover can be defined by requiring that disjoint finite sets be separated by a member of the cover.  We call a cover ${\mathcal U}$ of a space $X$ an \emph{$\iota $-cover} if for every pair of disjoint finite sets $F, G\subseteq X$, there is a member $U\in {\mathcal U}$ such that $F\subseteq U$ and $G\cap U =\emptyset $.  Notice that every space with a countable network is an $\epsilon $-space.  Furthermore, we will show that every $T_2$ space with a countable network has the property that every open $\iota $-cover has a countable refinement that is also an $\iota $-cover. Hence, we call such spaces with this property $\iota $-spaces.  The motivation for these definitions of $\iota$-cover and $\iota $-space arose when the third named author was trying to make the example in \cite{SS2} zero dimensional to solve the D-space problem.

We will explore the relationship between $\iota $-spaces and countable network weight, providing a ZFC example of a regular $\iota $-space with no countable network.  We also investigate the relationship between $\epsilon $-spaces and $\iota $-spaces, determining an additional property that makes them equivalent.  We use the notion of an $\iota $-cover to construct a hereditarily $\epsilon$-space whose square is not hereditarily Lindel\" of.  Finally, we give an example of a non D-space that has a countable open $\iota $-cover.

\section{Preliminaries}
\begin{definition}
A family of sets $\mathcal{U}$ is an $\omega$-cover of $X$ if for every $F \in [X]^{<\omega}$ there is $U \in \mathcal{U}$ such that $F \subset U$.
\end{definition}

\begin{definition}
A family of sets $\mathcal{U}$ is an \emph{$\iota$-cover} (\emph{$n$-ota cover}) of $X$ if for every $F, G \in [X]^{<\omega}$ ($F, G \in [X]^n$) such that $F \cap G=\emptyset$ there is a member $U \in \mathcal{U}$ such that $F \subset U$ and $G \cap U=\emptyset$.
\end{definition}

In the following proposition we collect a few trivial facts about $\iota$-covers and their relationship with $\omega$-covers.

\begin{proposition}{\ }
\begin{enumerate}
\item Every $\iota$-cover is an $\omega$-cover.
\item Any open $\omega$-cover of a $T_1$ topological space has an open refinement that is an $\iota$-cover.
\item Any fattening of an $\iota$-cover is an $\omega$-cover.
\end{enumerate}
\end{proposition}

\begin{definition}
We call a space $X$ an $\epsilon$-space if every open $\omega$-cover of $X$ has a countable $\omega$-subcover.
\end{definition}

\begin{definition} \label{defiotaspace}
We call a space $X$ an \emph{$\iota$-space} (\emph{$n$-ota space}) if every open $\iota$-cover ($n$-ota cover) of $X$ has a countable refinement which is an $\iota$-cover ($n$-ota cover).
\end{definition}

\begin{remark}
In Definition $\ref{defiotaspace}$ we used \emph{refinement} rather than \emph{subcover} because the class of spaces where every $\iota$-cover has a countable $\iota$-subcover coincides with the class of countable spaces. Indeed if $X$ is uncountable and $T_1$ then $\{X \setminus F : F \in [X]^{<\omega} \}$ is an $\iota$-cover without a countable $\iota$-subcover.
\end{remark}

While every space $X$ has a countable open $\omega$-cover (simply consider  $\{X\}$), not all spaces have countable $\iota$-covers.

\begin{example} \label{noiotacover}
A compact $T_2$ space of size $\omega_1$ without a countable $\iota$-cover.
\end{example}

\begin{proof}
Let $X=D \cup \{p\}$ be the one-point compactification of a discrete set of size $\aleph_1$, where $p$ is the unique non-isolated point. Suppose by contradiction that $X$ has a countable $\iota$-cover $\mathcal{U}$ and let $\mathcal{U}_p=\{U \in \mathcal{U}: p \in U \}$. The set $\mathcal{U}_p$ is countable and every element of $\mathcal{U}_p$ is a cofinite set. Therefore, the set $\bigcap \mathcal{U}_p$ is uncountable and hence we can fix distinct points $x, y \in \bigcap \mathcal{U}_p$. But then $\mathcal{U}$ has no element containing $\{p,x\}$ and missing $\{y\}$. Therefore $\mathcal{U}$ is not an $\iota$-cover.
\end{proof}

In view of Example $\ref{noiotacover}$ it makes sense to consider the following class of spaces.

\begin{definition}
We call a space $X$ an $\iota_w$-space if it has a countable open $\iota$-cover.
\end{definition}

Every $\iota$-space is certainly an $\iota_w$-space, but the converse is far from being true.

\begin{proposition} \label{bound}
Let $X$ be an $\iota_w$-space. Then $|X| \leq \mathfrak{c}$.
\end{proposition}

\begin{proof}
Let $\mathcal{U}$ be a countable open $\iota$-cover for $X$. Define a map $f: X \to [\mathcal{U}]^\omega$ as follows: $f(x)=\{U \in \mathcal{U}: x\in U \}$. Since $\mathcal{U}$ is an $\iota$-cover, $f$ is a one-to-one map. Therefore $|X| \leq |[\mathcal{U}]^\omega|=\mathfrak{c}$.
\end{proof}

\begin{corollary}
The discrete space of size $\kappa$ is an $\iota_w$-space if and only if $\kappa \leq \mathfrak{c}$.
\end{corollary}

\begin{proof}
If $\kappa \leq \mathfrak{c}$ fix a separable metric topology $\tau$ on $\kappa$. Any $\iota$-refinement of $\tau$ provides an open $\iota$-cover of $\kappa$ with the discrete topology. The converse follows from Proposition $\ref{bound}$.
\end{proof}

There is, however, a natural relationship between $\iota$-spaces and $\iota_w$-spaces.

\begin{theorem}\label{epsilon+weakiota}
A space $X$ is an $\iota$-space if and only if it is an $\epsilon$-space and an $\iota_w$-space.
\end{theorem}

\begin{proof}
The direct implication is trivial. To prove the converse implication, fix a countable $\iota$-cover $\mathcal{C}$ for $X$ and let $\mathcal{U}$ be any open $\iota$-cover.  Then $\mathcal{U}$ is also an $\omega$-cover, and since $X$ is an $\epsilon$-space we can find a countable $\omega$-subcover $\mathcal{V}$ of $\mathcal{U}$. The set $\{U \cap V: U \in \mathcal{C}, V \in \mathcal{V} \}$ is then a countable $\iota$-refinement of $\mathcal{U}$.
\end{proof}

There are hereditarily Lindel\"of spaces which are $\iota_w$-spaces, but not $\iota$-spaces. One such example is the Sorgenfrey line. Indeed, since its topology is a refinement of the topology of the real line, it has a countable $\iota$-cover, but its square is not Lindel\"of and hence it's not an $\iota$-space. Note that by Theorem $\ref{epsilon+weakiota}$ if $X$ is a subspace of the Sorgenfrey Line, then $X$ is an $\epsilon$-space if and only if $X$ is an $\iota$-space.

\begin{theorem}
Let $X$ be a Tychonoff space such that $C_p(X)$ is separable and has countable tightness. Then $X$ is an $\iota$-space.
\end{theorem}

\begin{proof}
From \cite{A}, $C_p(X)$ has countable tightness if and only if $X$ is an $\epsilon$-space and $C_p(X)$ is separable if and only if $X$ has a one-to-one continuous map onto a separable metrizable space. It's easy to see that this last condition is equivalent to $X$ having a coarser second-countable topology. But this easily implies that $X$ has a countable $\iota$-cover, that is, $X$ is an $\iota_w$-space.
\end{proof}

\begin{corollary}
Let $X$ be a Tychonoff space. Suppose $C_p(X)$ is hereditarily separable. Then $X$ is an $\iota$-space.
\end{corollary}

\begin{proposition}{\ } \label{subspaceprop}
\begin{enumerate}
\item \label{noniotaclosed} Let $(X, \tau)$ be an $\iota$-space, then every closed subspace is an $\iota$-space.  
\item \label{noniotaany} Let $(X, \tau)$ be an $\iota_w$-space, then every subspace is an $\iota_w$-space.
\end{enumerate}
\end{proposition}

\begin{proof}
To prove ($\ref{noniotaclosed}$) suppose $Y$ is a closed subspace of the $\iota$-space $X$. Fix an open cover $\mathcal{U}_Y$ of $Y$. Let $\mathcal{U}=\{U \in \tau: U \cap Y \in \mathcal{U}_Y \}$ and $\mathcal{U}^Y=\{U \in \tau: U \cap Y=\emptyset \}$. 

Let $$\mathcal{V}=\{(U \setminus F) \cup V: U \in \mathcal{U}, F \in [Y]^{<\omega} , V \in \mathcal{U}^Y \}$$

Then $\mathcal{V}$ is an $\iota$-cover for the whole space $X$ and the trace of any countable $\iota$-refinement of $\mathcal{V}$ on $Y$ is a countable $\iota$-refinement of $\mathcal{U}_Y$.

The proof of ($\ref{noniotaany}$) is similar and even easier.
\end{proof}

\begin{corollary}
Let $\{X_i: i \in I\}$ be a family of spaces, where $|X_i| \geq 2$ and $|I| \geq \aleph_1$. Then $\prod_{i \in I} X_i$ is not an $\iota_w$-space.
\end{corollary}

\begin{proof}
Simply note that $\prod_{i \in I} X_i$ contains a copy of $2^{|I|}$, which in turn contains a copy of the one-point compactification of a discrete space of size $|I|$ and that this space is not an $\iota_w$-space. 
\end{proof}

\begin{theorem} \label{productiw}
Let $\{X_i: i < \omega \}$ be a countable family of $\iota_w$-spaces. Then $X:=\prod_{i <\omega} X_i$ is an $\iota_w$-space.
\end{theorem}

\begin{proof}
Let $\mathcal{U}_i$ be a countable open $\iota$-cover for $X_i$. Consider two disjoint finite subsets $F$ and $G$ of $X$. For every $(x_n)_{n<\omega} \in F$ let $U_{x_i} \in \mathcal{U}_i$ be an open set containing $x_i$ and missing $\{z \in \pi_{X_i}(F \cup G): z \neq x_i \}$. Let now $U= \bigcup \{\prod_{i <\omega} U_{x_i}: (x_i)_{i<\omega} \in F \}$. Then $U$ contains $F$ and misses $G$. Indeed if $G \cap U$ were non-empty then there would be $(x_n)_{n<\omega} \in F$ such that $(\prod_{n<\omega} U_{x_n}) \cap G \neq \emptyset$, but this contradicts the definition of $U_{x_n}$ for $n<\omega$.

It follows that $\{\bigcup \{\prod_{i<\omega} U_i: U_i \in \mathcal{F}_i\}: \mathcal{F}_i \in [\mathcal{U}_i]^{<\omega} \}$ is a countable open $\iota$-cover for $\prod_{i<\omega} X_i$.
\end{proof}

\begin{corollary}
Let $\{X_i: i \in I\}$ be a family of $\iota_w$-spaces. Then $\prod_{i \in I} X_i$ is an $\iota_w$-space if and only if $|I| \leq \omega$.
\end{corollary}

\section{Countable Network Weight}
It is known that $T_2$ spaces with a countable network are $\epsilon$-spaces, so using Theorem \ref{epsilon+weakiota} we have the following.

\begin{theorem}
 Every $T_2$ space with a countable network is an $\iota$-space.
\end{theorem}

The converse is not true. We are going to present three counterexamples. The first one has the advantage of being simpler, the second one has the advantage of being regular, and the third one is only consistent, but we present it anyway, because the techniques used in verifying its properties might have independent interest.

\begin{example}
There is a $T_2$ $\iota$-space without a countable network.
\end{example}

\begin{proof}
Let $\mathbb{R}_c$ be the real line with the topology generated by sets of the form $U \setminus C$, where $U$ is a Euclidean open set and $C$ is a countable set of reals.

Suppose by contradiction that $\{N_n: n < \omega \}$ is a countable network for $\mathbb{R}_c$. Without loss of generality we can assume that $N_n$ is infinite for every $n$ and use this to inductively pick $x_n \in N_n \setminus \{x_i: i < n \}$. then $\mathbb{R}_c \setminus \{x_i: i < \omega \}$ is an open set not containing any element of $\{N_n: n < \omega \}$. It follows that $\mathbb{R}_c$ does not have a countable network.

Now $\mathbb{R}_c$ is a refinement of the Euclidean topology on $\mathbb{R}$ and hence it is both an $\epsilon $-space and an $\iota _w$-space. Therefore, by Theorem \ref{epsilon+weakiota}, $\mathbb{R}_c$ is an $\iota$-space.

\end{proof}

\begin{example}
There is a regular $\iota$-space without a countable network within the usual axioms of ZFC.
\end{example}

\begin{proof}

Let $X \subset \mathbb{R}$ be a subset of the reals. By Michael-type space $L(X)$ we mean the refinement of the usual topology on $\mathbb{R}$ obtained by isolating every point of $\mathbb{R} \setminus X$. By Theorem $\ref{epsilon+weakiota}$ every Michael-type space which is an $\epsilon$-space is also an $\iota$-space. It's easy to see, that if $X$ is a Bernstein set (that is, a set which hits every uncountable closed set of the real line along with its complement), then $L(X)$ is Lindel\"of, and Lawrence \cite{L} proved that there is in ZFC a Bernstein set $X \subset \mathbb{R}$ such that $L(X)$ is an $\epsilon$-space.  The techniques used to construct the Bernstein set originated in \cite{P} and Burke gives the details of the construction in \cite{B}.

\end{proof}

The next construction preceded Theorem \ref{epsilon+weakiota}, but we include it because it may be of independent interest.  It gives a recursive construction of an $\iota $-space.

\begin{example}[CH]\label{Michael}
There is a Michael space, $M_X$, that is an $\iota $-space.
\end{example}

\begin{proof}
For convenience, call ${\mathcal U}$ an open finite union (ofu)-$\iota $-cover of ${\mathbb Q}$ if 
\begin{enumerate}
\item $\forall U\in {\mathcal U}$, $U=\bigcup _{i<n}I_i$ where $n\in \omega $, $I_i=(p_i,q_i)$, $p_i,q_i\in {\mathbb Q}$.
\item ${\mathbb Q}\subseteq \bigcup {\mathcal U}$
\item $\forall F,G\in [{\mathbb Q}]^{<\omega }$ such that $F\cap G=\emptyset $, $\exists U=\bigcup _{i<n}I_i\in {\mathcal U}$ such that $F\subseteq U$, $\bar{I}_i\cap G=\emptyset $, $\forall i<n$.
\end{enumerate}
Let $\{{\mathcal U}_\alpha :\alpha <\omega _1\}$ enumerate all (ofu)-$\iota $-covers of ${\mathbb Q}$.  Define by recursion $X=\{x_\alpha :\alpha <\omega _1\}$ so that $\forall \alpha <\omega _1$, IH($\alpha )$ holds, where
\\
\\
IH($\alpha $):  $\forall \beta <\alpha $, ${\mathcal U}_\beta $ is an $\iota $-cover of ${\mathbb Q}\cup \{x_\xi :\beta <\xi <\alpha \}$. 
\\
\\ 
Let $x_0\in {\mathbb R}\setminus {\mathbb Q}$.
\\
Fix $\alpha <\omega _1$ and suppose $\{x_\xi :\xi <\alpha \}$ have been defined.  
\\
We must choose $x_\alpha $ so that IH($\alpha +1$) is satisfied.  That is, ${\mathcal U}_\beta $ is an $\iota $-cover of ${\mathbb Q}\cup \{x_\xi :\beta <\xi \leq \alpha \}$, $\forall \beta \leq \alpha $.

\begin{notation}
For ${\mathcal U}\in \{{\mathcal U}_\beta :\beta \leq \alpha \}$, let ${\mathbb Q}_{{\mathcal U}}={\mathbb Q}\cup \{x_\xi :\beta <\xi <\alpha \}$ where ${\mathcal U}={\mathcal U}_\beta $.
\end{notation}
Let ${\mathcal T}=\{({\mathcal U},F,G)\in \{{\mathcal U}_\beta :\beta \leq \alpha \}\times [{\mathbb Q}\cup \{x_\xi :\xi <\alpha \}]^{<\omega }\times [{\mathbb Q}\cup \{x_\xi :\xi <\alpha \}]^{<\omega }: F,G\in [{\mathbb Q}_{\mathcal U}]^{<\omega }, F\cap G=\emptyset \}$.  Enumerate ${\mathcal T}=\{({\mathcal U}_{\alpha n},F_n,G_n):n\in \omega \}$.  For $n\in \omega $, let $\beta _n \leq \alpha $ such that ${\mathcal U}_{\alpha n}={\mathcal U}_{\beta _n}$.  Then, $F_n, G_n \in [{\mathbb Q}\cup \{x_\xi :\beta _n<\xi <\alpha \}]^{<\omega }$ and by IH($\alpha )$, ${\mathcal U}_{\alpha n}$ is an $\iota $-cover of ${\mathbb Q}\cup \{x_\xi :\beta _n<\xi <\alpha \}$.
\\
\\
Build sequences $\{q_n:n\in \omega \}\subseteq {\mathbb Q}$, $\{U_n^i:n\in \omega, i<2\}$, $\{I_n:n\in \omega \}$ and $\{V_n:n\in \omega \}$ such that 
\begin{enumerate}
\item $U_n^i\in {\mathcal U}_{\alpha n} \forall i<2$, $n\in \omega $.
\item $I_n$, $V_n$ open intervals.
\item $q_0\in {\mathbb Q}\setminus (F_0\dot{\cup }G_0)$ and $q_n\in V_{n-1}\cap {\mathbb Q}\setminus (F_n\dot{\cup }G_n)$, $\forall n\geq 1$.
\item $F_n\cup \{q_n\}\subseteq U_n^0$, $U_n^0\cap G_n=\emptyset $ and $F_n\subseteq U_n^1$, $U_n^1\cap (G_n\cup \{q_n\})=\emptyset $.
\item $q_0\in I_0\subseteq U_n^0$ and $q_n\in I_n\subseteq V_{n-1}\cap U_n^0$, $\forall n\geq 1$.
\item $\overline{V}_n\subseteq I_n\setminus (U_n^1\cup \{q_n\})$ such that diam$(V_n)<\frac{1}{n}$ ($\forall n\geq 1)$
\end{enumerate}
Let $q_0\in {\mathbb Q}\setminus (F_0\dot{\cup }G_0)$.  Then $F_0\cup \{q_0\}, G_0\in [{\mathbb Q}\cup \{x_\xi :\beta _0<\xi <\alpha \}]^{<\omega }$ so let $U_0^0=\bigcup _{i<k_0}I_i \in {\mathcal U}_{\alpha 0}$ such that $F_0\cup \{q_0\}\subseteq U_0^0$, $G_0\cap U_0^0=\emptyset $.  Let $I_0\in \{I_i:i<k_0\}$ such that $q_0\in I_0$.  Also, $F_0,G_0\cup \{q_0\}\in [{\mathbb Q}\cup \{x_\xi :\beta _0<\xi <\alpha \}]^{<\omega }$ so let $U_0^1=\bigcup _{i<m_0}I_i$ such that $F_0\subseteq U_0^1$, $(G_0\cup \{q_0\})\cap U_0^1=\emptyset $.  Let $V_0$ be an open interval such that $\overline{V}_0\subseteq I_0\setminus (U_0^1\cup \{q_0\})$.
\\
Fix $n\in \omega $ and suppose $\{q_m:m<n\}$, $\{U_m^i:i<2,m<n\}$, $\{I_m:m<n\}$ and $\{V_m:m<n\}$ have been defined.
\\
Let $q_n\in V_{n-1}\cap {\mathbb Q}\setminus (F_n\dot{\cup }G_n)$.  Then $F_n\cup \{q_n\}, G_n\in [{\mathbb Q}\cup \{x_\xi :\beta _n<\xi <\alpha \}]^{<\omega }$ so let $U_n^0=\bigcup _{i<k_n}I_i \in {\mathcal U}_{\alpha n}$ such that $F_n\cup \{q_n\}\subseteq U_n^0$, $G_n\cap U_n^0=\emptyset $.  Let $I_n\subseteq V_{n-1}\cap U_n^0$ be an open interval such that $q_n\in I_n$.  Also, $F_n,G_n\cup \{q_n\}\in [{\mathbb Q}\cup \{x_\xi :\beta _n<\xi <\alpha \}]^{<\omega }$ so let $U_n^1=\bigcup _{i<m_n}I_i$ such that $F_n\subseteq U_n^1$, $(G_n\cup \{q_n\})\cap U_n^1=\emptyset $.  Let $V_n$ be an open interval of diameter $<\frac{1}{n}$ such that $\overline{V}_n\subseteq I_n\setminus (U_n^1\cup \{q_n\})$.
\\
Let $x_\alpha \in {\mathbb R}\setminus {\mathbb Q}$ such that $\bigcap _{n\in \omega }\overline{V}_n=\{x_\alpha \}$.
\\
To see IH$(\alpha +1)$ is satisfied, let $\beta \leq \alpha $ and notice ${\mathcal U}_\beta $ is an $\iota $-cover of ${\mathbb Q}\cup \{x_\xi :\beta <\xi \leq \alpha \}$:  Let $F,G\in [{\mathbb Q}\cup \{x_\xi :\beta <\xi \leq \alpha \}]^{<\omega }$ such that $F\cap G=\emptyset $.  If $x_\alpha \in F$, let $F^\shortmid =F\setminus \{x_\alpha \}$ and $m\in \omega $ such that $({\mathcal U}_\beta , F^\shortmid , G)=({\mathcal U}_{\alpha m},F_m,G_m)$.  Then, $U_m^0\in {\mathcal U}_{\alpha m} $ such that $F_m\subseteq U_m^0, G_m\cap U_m^0=\emptyset $ and $x_\alpha \in V_m\subseteq U_m^0$.  Therefore, $U_m^0\in {\mathcal U}_{\beta} $ such that $F\subseteq U_m^0$ and $G\cap U_m^0=\emptyset $.  If $x_\alpha \in G$, let $G^\shortmid =G\setminus \{x_\alpha \}$ and $k\in \omega $ such that $({\mathcal U}_\beta , F, G^\shortmid )=({\mathcal U}_{\alpha k},F_k,G_k)$.  Then, $U_k^1\in {\mathcal U}_{\alpha k} $ such that $F_k\subseteq U_k^1, G_k\cap U_k^1=\emptyset $ and $x_\alpha \in V_k\subseteq I_k\setminus (U_k^1\cup \{q_k\})$.  Therefore, $U_k^1\in {\mathcal U}_{\beta} $ such that $F\subseteq U_k^1$ and $G\cap U_k^1=\emptyset $. 

Therefore, by construction, ${\mathcal U}_\beta $ is an $\iota $-cover of ${\mathbb Q}\cup \{x_\xi :\xi >\beta \}$, $\forall \beta <\omega _1$.  So, ${\mathcal U}_\beta $ is an $\iota $-cover of a tail of ${\mathbb Q}\cup X$.

Let $M_X={\mathbb Q}\cup X$ with the Michael topology (usual basic open neighbourhoods for ${\mathbb Q}$ and isolate points of $X$).  Let ${\mathcal U}$ be an open $\iota $-cover of $M_X$.
\begin{notation}
For $F,G\in [X]^{<\omega }$ such that $F\cap G=\emptyset $, let ${\mathcal U}_{FG}=\{U\in {\mathcal U}:F\subseteq U, U\cap G=\emptyset \}$.
\end{notation}
\begin{claim}
For any $F,G \in [X]^{<\omega }$ such that $F\cap G =\emptyset $, $\exists \alpha _{FG}<\omega _1$ such that ${\mathcal U}_{\alpha _{FG}}\prec {\mathcal U}_{FG}$.
\end{claim} 
Note that, ${\mathcal U}_{FG}$ is an $\iota $-cover of ${\mathbb Q}\cup X\setminus (F\cup G)$.  So, for each $F^\shortmid , G^\shortmid \in [{\mathbb Q}]^{<\omega }$ such that $F^\shortmid \cap G^\shortmid =\emptyset $, let $U(F^\shortmid ,G^\shortmid )\in {\mathcal U}_{FG}$ such that $F^\shortmid \subseteq U(F^\shortmid ,G^\shortmid )$ and $U(F^\shortmid ,G^\shortmid )\cap G^\shortmid =\emptyset $.  For $x\in F^\shortmid $, let $p_{x(F^\shortmid G^\shortmid )}, q_{x(F^\shortmid G^\shortmid )} \in {\mathbb Q}$ such that $x\in I_{x(F^\shortmid G^\shortmid )}=(p_{x(F^\shortmid G^\shortmid )}, q_{x(F^\shortmid G^\shortmid )})\subseteq U(F^\shortmid ,G^\shortmid )$ but $\bar{I}_{x(F^\shortmid G^\shortmid )}\cap G^\shortmid =\emptyset $.  Let $V_{F^\shortmid G^\shortmid }=\bigcup _{x\in F^\shortmid }I_{x(F^\shortmid G^\shortmid )}$.  Then ${\mathcal V}=\{V_{F^\shortmid G^\shortmid }:F^\shortmid , G^\shortmid \in [{\mathbb Q}]^{<\omega }, F^\shortmid \cap G^\shortmid =\emptyset \}$ is an (ofu)-$\iota $-cover of ${\mathbb Q}$ that refines ${\mathcal U}_{FG}$.  So, let $\alpha _{FG}<\omega _1$ such that ${\mathcal V}={\mathcal U}_{\alpha _{FG}}$.
\\
\\
By a closing off argument, let $\bar{\alpha }<\omega _1$ such that ${\mathcal U}_{\bar{\alpha }}\prec {\mathcal U}_{FG}$, $\forall F,G\in [\{x_\xi :\xi \leq \bar{\alpha }\}]^{<\omega }$ such that $F\cap G=\emptyset $.
\begin{claim}
${\mathcal U}^\shortmid =\{U\cup F:U\in {\mathcal U}_{\bar{\alpha }}, F\in [\{x_\xi :\xi \leq \bar{\alpha }\}]^{<\omega }\}$ is a countable open refinement of ${\mathcal U}$ that is an $\iota $-cover.
\end{claim}  
To see ${\mathcal U}^\shortmid \prec {\mathcal U}$, let $U\cup F\in {\mathcal U}^\shortmid $.  Then $U\in {\mathcal U}_{\bar{\alpha }}$, $F\in [\{x_\xi :\xi \leq \bar{\alpha }\}]^{<\omega }$ so let $G\in [\{x_\xi :\xi \leq \bar{\alpha }\}\setminus F]^{<\omega }$ and since ${\mathcal U}_{\bar{\alpha }}\prec {\mathcal U}_{FG}$, let $U^\shortmid \in {\mathcal U}_{FG}$ such that $U\subseteq U^\shortmid $.  Then, $U^\shortmid \in {\mathcal U}$ such that $U\cup F\subseteq U^\shortmid $.

To see ${\mathcal U}^\shortmid $ is an $\iota $-cover, let $F^\shortmid ,G^\shortmid \in [M_X]^{<\omega }$ such that $F^\shortmid \cap G^\shortmid =\emptyset $.  Let $F_1^\shortmid =F^\shortmid \cap \{x_\xi :\xi \leq \bar{\alpha }\}$, $F_2^\shortmid =F^\shortmid \cap ({\mathbb Q}\cup \{x_\xi :\xi >\bar{\alpha }\})$, $G_1^\shortmid =G^\shortmid \cap \{x_\xi :\xi \leq \bar{\alpha }\}$, $G_2^\shortmid =G^\shortmid \cap ({\mathbb Q}\cup \{x_\xi :\xi >\bar{\alpha }\})$.  Then, $F_2^\shortmid ,G_2^\shortmid \in [{\mathbb Q}\cup \{x_\xi :\xi >\bar{\alpha }\}]^{<\omega }$ such that $F_2^\shortmid \cap G_2^\shortmid =\emptyset $.  So let $U\in {\mathcal U}_{\bar{\alpha }}$ such that $F_2^\shortmid \subseteq U$ and $U\cap G_2^\shortmid =\emptyset $.  Then, $U\cup F_1^\shortmid \in {\mathcal U}^\shortmid $ such that $F^\shortmid \subseteq U\cup F_1^\shortmid $ and $(U\cup F_1^\shortmid )\cap G^\shortmid =\emptyset $.
\end{proof}

\begin{remark}
We constructed $M_X$ so that any open $\iota $-cover of $M_X$ has a countable open refinement that $\iota $-covers a tail of $M_X$, which is enough to show that $M_X$ is an $\iota $-space since the countable open refinement ${\mathcal U}^\shortmid $ of ${\mathcal U}$ that $\iota $-covers $M_X$ is defined from an $\iota $-cover of a tail or an almost $\iota $-cover.  This leads us to our next definition and some useful facts.
\end{remark}

\begin{definition}\label{almostiota}
A space $X$ is {\it almost}-$\iota $ if for every open $\iota $-cover ${\mathcal U}$ of $X$, there is a countable open refinement ${\mathcal V}$ of ${\mathcal U}$ and $A\in [X]^\omega $ such that ${\mathcal V}$ is an $\iota $-cover of $X\setminus A$.
\end{definition}
\begin{note}
Almost-$\iota $ is closed hereditary.
\end{note}
\begin{lemma}\label{almostiotalem}
If $X$ is almost-$\iota $ and has points regular $G_\delta $ then $X$ is an $\iota $-space.
\end{lemma}
Before proving Lemma \ref{almostiotalem}, we need the following:
\begin{lemma}
If $X$ is almost-$\iota $ and has points regular $G_\delta $ then $X\setminus F$ is almost-$\iota, $ $\forall F\in [X]^{<\omega }$.
\end{lemma}
\begin{proof}
Fix $F\in [X]^{<\omega }$ and let ${\mathcal U}$ be any open $\iota $-cover of $X\setminus F$.  Since $X$ has points regular $G_\delta $, let $U_n\subseteq X$ be open such that $F\subseteq U_n$, $\forall n\in \omega $ and $F=\bigcap\overline{U_n}$.  For $n\in \omega $, let ${\mathcal U}_n=\{U\cap X\setminus U_n:U\in {\mathcal U}\}$, which is an open $\iota $-cover of $X\setminus U_n$.  Thus, since $X\setminus U_n$ is almost-$\iota $ (being a closed subspace of an almost-$\iota $ space) let ${\mathcal U}_n^\shortmid $ be a countable open refinement of ${\mathcal U}_n$ and $A_n\in [X\setminus U_n]^{\omega }$ such that ${\mathcal U}_n^\shortmid $ is an $\iota $-cover of $(X\setminus U_n)\setminus A_n$.  Finally, $\forall n\in \omega $, let ${\mathcal V}_n=\{V\setminus \overline{U_n}:V\in {\mathcal U}_n^\shortmid \}$.  The following claim finishes the proof:
\begin{claim}
$\bigcup _{n\in \omega }{\mathcal V}_n$ is a countable open refinement of ${\mathcal U}$ that is an $\iota $-cover of $(X\setminus F)\setminus (\bigcup _{n\in \omega }A_n)$.
\end{claim}
Let $F^\shortmid ,G^\shortmid \in [(X\setminus F)\setminus (\bigcup _{n\in \omega }A_n)]^{<\omega }$ such that $F^\shortmid \cap G^\shortmid =\emptyset $.  Then, $(F^\shortmid \cup G^\shortmid )\cap F=\emptyset $ and $(F^\shortmid \cup G^\shortmid )\cap \bigcup _{n\in \omega }A_n=\emptyset $.  So, $\exists k\in \omega $ such that $(F^\shortmid \cup G^\shortmid )\cap \overline{U_k}=\emptyset $ and $(F^\shortmid \cup G^\shortmid )\cap A_k=\emptyset $.  Thus, $F^\shortmid ,G^\shortmid \in [(X\setminus U_k)\setminus A_k]^{<\omega }$ such that $F^\shortmid \cap G^\shortmid =\emptyset $.  So, let $U \in {\mathcal U_k}^\shortmid $ such that $F^\shortmid \subseteq U$ and $U\cap G^\shortmid =\emptyset $.  Then $U\setminus \overline{U_k}\in {\mathcal V}_k$ such that $F^\shortmid \subseteq U\setminus \overline{U_k}$ and $(U\setminus \overline{U_k})\cap G^\shortmid =\emptyset $. 
\end{proof}

\begin{proof}[Proof of Lemma \ref{almostiotalem}]
Let ${\mathcal U}$ be any open $\iota $-cover of $X$.  Let ${\mathcal M}$ be a countable elementary submodel of $H_{\theta }$ (for $\theta $ large enough) such that ${\mathcal U}, (X,\tau )\in {\mathcal M}$.
\begin{claim}
${\mathcal V}=\{V\in {\mathcal M}\cap \tau :V\subseteq U$ for some $U\in {\mathcal U}\}$ is a countable open refinement of ${\mathcal U}$ that is an $\iota $-cover of $X$.
\end{claim}
Let $F^\shortmid ,G^\shortmid \in [X]^{<\omega }$ such that $F^\shortmid \cap G^\shortmid =\emptyset $.  Let $F=F^\shortmid \cap {\mathcal M}$ and $G=G^\shortmid \cap {\mathcal M}$.  By elementarity, ${\mathcal M}\models (X\setminus E$ is almost-$\iota $, $\forall E\in [X]^{<\omega })$.  Thus, since $F,G\in {\mathcal M}$, ${\mathcal M}\models X\setminus (F\cup G)$ is almost-$\iota $.  Notice ${\mathcal U}_{FG}=\{U\in {\mathcal U}:F\subseteq U, U\cap G=\emptyset \}\in {\mathcal M}$ is an open $\iota $-cover of $X\setminus (F\cup G)$.  So, let ${\mathcal V}_{FG}\in {\mathcal M}$ be a countable open refinement of ${\mathcal U}_{FG}$ and $A_{FG}\in [X]^{\omega }\cap {\mathcal M}$ such that ${\mathcal V}_{FG}$ is an $\iota $-cover of $(X\setminus (F\cup G))\setminus A_{FG}$.  Then, ${\mathcal V}_{FG},A_{FG}\subseteq {\mathcal M}$. 
Thus $F^\shortmid \setminus F, G^\shortmid \setminus G\in [(X\setminus (F\cup G))\setminus A_{FG}]^{<\omega }$ 
such that $F^\shortmid \setminus F \cap G^\shortmid \setminus G =\emptyset $.  So let $V\in {\mathcal V}_{FG}$ such that $F^\shortmid \setminus F\subseteq V $ and $V\cap G^\shortmid \setminus G=\emptyset $.  Since ${\mathcal V}_{FG}$ refines ${\mathcal U}_{FG}$, $\exists U\in {\mathcal U}_{FG}(V\subseteq U)$.  So, by elementarity, let $U\in {\mathcal U}_{FG}\cap {\mathcal M}$ such that $V\subseteq U$.  Also, ${\mathcal M}\models (x$ is regular $G_\delta $, $\forall x\in X$), so in particular, since $F\in {\mathcal M}$ is finite, let $U_n\in {\mathcal M}$ such that $F\subseteq U_n$, $\forall n\in \omega $ and $F=\bigcap _{n\in \omega }\overline{U_n}$.  Then, since $G^\shortmid \cap F=\emptyset$, $\exists n\in \omega $ such that $G^\shortmid \cap \overline{U_n}=\emptyset $ and 
$V\cup (U_n\cap U)\in {\mathcal V}$ such that $F^\shortmid \subseteq V\cup (U_n\cap U)$ and $G^\shortmid \cap (V\cup (U_n\cap U))=\emptyset $.  
\end{proof} 

Returning again to the relationship with countable network weight, we have seen some (consistent) counterexamples, but restricting ourselves to the hereditary property raises the natural question.

\begin{question}
Is every hereditarily $\iota$-space a space with a countable network?
\end{question}

\section{$\epsilon $-spaces}
Theorem \ref{epsilon+weakiota} provides us with an instance when $\epsilon $-spaces and $\iota $-spaces are equivalent.  We investigate what additional characteristics can be placed on an $\epsilon $-space to ensure it is an $\iota $-space. 

\begin{definition}
Let $\mathcal{U}$ be a cover of a space $X$. We say that $\mathcal{U}$ is a \emph{regular 1-ota cover} if for every $x\neq y \in X$ there is $U \in \mathcal{U}$ such that $x \in U$ and $y \notin \overline{U}$.
\end{definition}

\begin{lemma} \label{lemregular}
Let $\mathcal{U}$ be an $\iota$-cover of the regular space $X$. Then $\mathcal{U}$ has a regular 1-ota refinement.
\end{lemma}

\begin{proof}
For every $x\neq y \in X$ choose $U(x,y) \in \mathcal{U}$ such that $x \in U(x,y)$ and $y \notin U(x,y)$.  Now let $V(x,y)$ be an open set such that $x \in V(x,y) \subset \overline{V(x,y)} \subset U(x,y)$. Then $\mathcal{V}=\{V(x,y): x\neq y \in X \}$ is a regular 1-ota refinement of $\mathcal{U}$.
\end{proof}

\begin{lemma}\label{tool}
Let $X$ be a regular space such that $X^2 \setminus \Delta$ is Lindel\"of. Then $X$ is 1-ota.
\end{lemma}

\begin{proof}
Let $\mathcal{U}$ be a 1-ota cover for $X$ without a countable 1-ota refinement. Let $\mathcal{V}$ be a regular 1-ota refinement of $\mathcal{U}$ having minimal size $\kappa \geq \omega_1$. 

Fix an enumeration $\{V_\alpha: \alpha < \kappa \}$ of $\mathcal{V}$ and let $\mathcal{V}_\alpha:=\{V_\beta: \beta \leq \alpha \}$ and:

$$A(\mathcal{V}_\alpha):=\{ (x,y) \in X^2 \setminus \Delta: (\forall U \in \mathcal{V_\alpha}) ((x \in U \wedge y \in \overline{U}) \vee (x \in \overline{U} \wedge y \in U) \vee (\{x,y\} \cap U=\emptyset ) \}.$$

\noindent {\bf Claim.} $A(\mathcal{V}_\alpha) \neq \emptyset$ and $A(\mathcal{V}_\alpha)$ is closed in $X^2 \setminus \Delta$ for every $\alpha < \omega_1$.

\begin{proof}[Proof of Claim] The fact that $A(\mathcal{V}_\alpha) \neq \emptyset$ follows from the assumptions about $\mathcal{U}$. To prove that $A(\mathcal{V}_\alpha)$ is closed, let $(x,y) \notin A(\mathcal{V}_\alpha) \cup \Delta$. Then we can find $U_x, U_y \in \mathcal{V}_\alpha$ such that $x \in U_x$, $y \notin \overline{U_x}$, $y \in U_y$ and $x \notin \overline{U_y}$. Now $(U_x \setminus \overline{U_y}) \times (U_y \setminus \overline{U_x})$ is an open neighbourhood of $(x,y)$ which misses $A(\mathcal{V}_\alpha)$.
\end{proof}
So $\{A(\mathcal{V}_\alpha): \alpha < \kappa \}$ is an uncountable decreasing sequence of non-empty closed subsets of the Lindel\"of space $X^2 \setminus \Delta$ and thus $\bigcap \{A(\mathcal{V}_\alpha): \alpha < \omega_1 \} \neq \emptyset$ which contradicts regularity of $\mathcal{V}$.
\end{proof}

Note that the fact that $X^2 \setminus \Delta$ is Lindel\"of implies that $X$ has a $G_\delta$ diagonal, whenever $X$ is regular. Indeed, for every $x \in X^2 \setminus \Delta$, let $U_x$ be an open neighbourhood of $x$ such that $\overline{U_x} \cap \Delta = \emptyset$. The family $\{U_x: x \in X^2 \setminus \Delta \}$ covers $X^2 \setminus \Delta$, and hence there is a countable set $C \subset X^2 \setminus \Delta$ such that $X^2 \setminus \Delta = \bigcup \{\overline{U_x}: x \in C \}$ and hence $\Delta = \bigcap_{x \in C} X^2 \setminus \overline{U_x}$, which proves that $\Delta$ is a $G_\delta$ subset of $X$.

The following lemma is not new. For example, the proof of a more general statement can be found in \cite{BG}. We nevertheless include a quick direct proof of it for the reader's convenience.

\begin{lemma}
Every countably compact 1-ota space $X$ is metrizable.
\end{lemma}

\begin{proof}
Let $\mathcal{U}$ be an $\iota$-cover of $X$. By Lemma $\ref{lemregular}$ we can assume that $\mathcal{U}$ is a regular $\iota$-cover. Let $\mathcal{V}$ be a countable $\iota$-refinement of $\mathcal{U}$ and let $\mathcal{B}=\{ X \setminus \overline{\bigcup \mathcal{F}}: \mathcal{F} \in [\mathcal{V}]^{<\omega}\}$. The set $\mathcal{B}$ is countable. We claim that $\mathcal{B}$ is a base of $X$, proving that $X$ is metrizable.

To see that let $U$ be an open set and $x \in U$. For every $y \in X \setminus U$ choose an open set $U_y \in \mathcal{V}$ such that $y \in U_y$ and $x \notin \overline{U_y}$. The countable set $\{U_y: y \in X \setminus U \}$ covers $X \setminus U$ so we can choose a finite set $F \subset X \setminus U$ such that $X \setminus U \subset \bigcup_{y \in F} U_y$. Then $x \in \bigcap_{y \in F} X \setminus \overline{U_y} \subset U$ and hence $\mathcal{B}$ is a base.
\end{proof}

\begin{corollary}
There are no countably compact strong $L$-spaces.
\end{corollary}

\begin{proof}
If $X^2$ is hereditarily Lindel\"of then $X$ is 1-ota and every countably compact 1-ota space is metrizable and thus separable.
\end{proof}

\begin{corollary}
Every compact space with a $G_\delta$ diagonal is metrizable.
\end{corollary}

\begin{proof}
If $X$ is a compact space with a $G_\delta$ diagonal then $X^2 \setminus \Delta$ is $\sigma$-compact. Thus $X$ is a compact 1-ota space and hence it's metrizable.
\end{proof}

Generalizing Lemma \ref{tool} provides us with a characterization that we are looking for.

Let $\Delta_n=\{(x_1, \dots, x_n) \in X^n: |\{x_1, x_2, \dots, x_n \}|<n\}$. Clearly, $\Delta_n$ is a closed subset of $X^n$.

\begin{theorem} \label{powersthm}
Assume $X$ is a regular space. If $X^{2n} \setminus \Delta_{2n}$ is Lindel\"of for every $n\in \omega $ then $X$ is an $\iota$-space.
 \end{theorem}

\begin{proof}
Suppose that $X^{2n} \setminus \Delta_{2n}$ is Lindel\"of. We prove that $X$ is $n$-ota. Indeed, let $\mathcal{U}$ be a $n$-ota cover without a countable $n$-ota refinement.  Let $\mathcal{V}$ be a regular $n$-ota refinement of $\mathcal{U}$ having minimal size $\kappa \geq \omega_1$. Enumerate $\mathcal{V}$ as $\{V_\alpha: \alpha < \aleph_1 \}$ and let $\mathcal{V}_\alpha=\{V_\beta: \beta \leq \alpha \}$ and $$A(\mathcal{V_\alpha})=\{(x_1, \dots x_{2n}) \in X^{2n} \setminus \Delta_{2n}: (\forall U \in \mathcal{V}_\alpha)(\{x_1, \dots x_n \} \subseteq U \wedge \{x_{n+1}, \dots , x_{2n}\} \cap \overline{U}\neq \emptyset)$$
$$ \vee(\{x_1, \dots, x_n\} \cap \overline{U} \neq \emptyset \wedge \{x_{n+1}, \dots x_{2n}\} \subseteq U) \vee (\{x_1, \dots, x_n \} \nsubseteq U \wedge \{x_{n+1}, \dots x_{2n}\} \nsubseteq U) \} $$ 
\vspace{.1in}

\noindent {\bf Claim.} $A(\mathcal{V}_\alpha)$ is closed.

\begin{proof}[Proof of Claim]
Let $(x_1, \dots, x_{2n}) \notin A(\mathcal{V}_\alpha) \cup \Delta_{2n}$. Then we can find sets $U_1$ and $U_2$ which are open in $x$ and such that $\{x_1, \dots, x_n\} \subset U_1$, $\{x_{n+1}, \dots x_{2n}\} \cap \overline{U_1}=\emptyset$, $\{x_{n+1}, \dots, x_{2n}\} \subset U_2$ and $\{x_1, \dots, x_n \} \cap \overline{U_2}=\emptyset$. Then $(U_1^n \setminus \overline{U_2^n}) \times (U_2^n \setminus \overline{U_1^n})$ is an open neighbourhood of $(x_1, \dots x_n, x_{n+1}, \dots, x_{2n})$ which misses $A(\mathcal{V}_\alpha)$. 
\end{proof}
So $\{A(\mathcal{V}_\alpha): \alpha < \kappa \}$ is an uncountable decreasing chain of closed sets in $X^{2n} \setminus \Delta_{2n}$, and hence it has non-empty intersection. This contradicts that $\mathcal{V}$ is a regular $n$-ota cover. So if $X^i \setminus \Delta_i$ is Lindel\"of for every $i<\omega$ then $X$ is $i$-ota for every $i<\omega$ and hence an $\iota$-space.

\end{proof}

\begin{corollary}
Every $\epsilon$-space with a $G_\delta$ diagonal is an $\iota$-space.
\end{corollary}

\begin{proof}
Suppose $X$ has a $G_\delta$ diagonal. Then $\Delta_n$ is a finite union of $G_\delta$ sets, and thus $G_\delta$. It follows that $X^n \setminus \Delta_n$ is a countable union of Lindel\"of spaces, and thus Lindel\"of. So $X$ is an $\iota$-space by Theorem \ref{powersthm}.
\end{proof}

\begin{corollary}
If $X^{2n} \setminus \Delta_{2n}$ is Lindel\"of for some $n$, then $X$ is Lindel\"of.
\end{corollary}

\begin{proof}
By the proof of Theorem $\ref{powersthm}$, $X$ is an $n$-ota space. But every $n$-ota space is Lindel\"of. Indeed, let $\mathcal{U}$ be an open cover for $X$. Let $\mathcal{V}$ be the set of all $n$-sized unions from $\mathcal{U}$. Let $\mathcal{G}$ be an $n$-ota refinement of $\mathcal{V}$ and $\mathcal{F}$ be a countable refinement of $\mathcal{G}$. Then $\mathcal{F}$ naturally induces a countable refinement of the original cover $\mathcal{U}$.
\end{proof}

Recall that a space $X$ is a \emph{Lindel\"of $\Sigma$-space} if it has a cover $\mathcal{C}$ by compact sets, and a countable family $\mathcal{N}$ of closed subsets of $X$ which is a \emph{network modulo $\mathcal{C}$}, that is, for every $C \in \mathcal{C}$ and every open set $U$ such that $C \subset U$ there is $N \in \mathcal{N}$ such that $C \subset N \subset U$. We will use this notion to provide an instance of when being an $\iota$-space and having a countable network are equivalent. The proof of the following theorem is similar to the proof that every Lindel\"of $\Sigma$-space is stable (see, for example, \cite{Tk}).

\begin{theorem}
Let $X$ be a regular Lindel\"of $\Sigma$-space. If $X$ is an $\iota$-space then $X$ has a countable network.
\end{theorem}

\begin{proof}
Let $\mathcal{C}$ be a cover of $X$ consisting of compact sets and $\mathcal{N}$ be a countable family which is a countable network for $X$ modulo $\mathcal{C}$. Since $X$ is regular, we can use Lemma $\ref{lemregular}$ to fix a countable regular $\iota$-cover $\mathcal{U}$ for $X$. Let $\mathcal{G}=\{\overline{U}: U \in \mathcal{U} \}$. We claim that the following family is a countable network for $X$:

$$\mathcal{B}=\left \{\bigcap \mathcal{F}: \mathcal{F} \in [\mathcal{G} \cup \mathcal{U}]^{<\omega} \right \}.$$

To see that, let $x \in X$ and $U$ be an open neighbourhood of $x$. Since $\mathcal{C}$ covers $X$, there is a $C \in \mathcal{C}$ such that $x \in C$. If $C \subset U$ then we can find $N \in \mathcal{N}$ such that $x \in C \subset N \subset U$ and we are done. Otherwise,  the set $K=C \setminus U$ is compact non-empty. For every $y \in K$ choose a set $G_y \in \mathcal{G}$ such that $y \in G$ but $x \notin G$. Then $\{G_y \cap K: y \in K \}$ has empty intersection, and hence, by compactness of $K$ there is a finite set $S \subset K$ such that $\bigcap_{y \in S} G_y \cap K = \emptyset$. Therefore, $\bigcap_{y \in  S} G_y \cap C \setminus U = \emptyset$, and hence $C \subset X \setminus (\bigcap_{y \in S} G_y \setminus U)$. But then there is an $N \in \mathcal{N}$ such that $C \subset N \subset X \setminus (\bigcap_{y \in S}  G_y \setminus U)$. Therefore $N \cap \bigcap_{y \in S} G_y \setminus U=\emptyset$ and hence $x \in N \cap \bigcap_{y \in S} G_y \subset U$, which is what we wanted, since $N \cap \bigcap_{y \in S} G_y \in \mathcal{B}$.

\end{proof}

\begin{question}
Is there a Lindel\"of Hausdorff $\Sigma$-space without a countable network which is an $\iota$-space?
\end{question}

Note that a Lindel\"of $\Sigma$-space which is an $\iota_w$-space is also an $\iota$-space, since countable products of Lindel\"of $\Sigma$-spaces are Lindel\"of $\Sigma$.

\section{L-spaces}
Since every $\iota $-space is an $\epsilon $-space and $\epsilon $-spaces are characterized by having all finite powers Lindel\" of, L-spaces are interesting spaces for us to consider.  It is conjectured that even in ZFC there is an L-space that is not even an $\epsilon $-space.  That is, of course, it is conjectured that Justin Moore's $L$-space has a finite power which is not Lindel\"of in ZFC.  Certainly this is consistently known.

\begin{proposition}\label{Lspace}
Consistently, every hereditarily Lindel\"of $\iota$-space is separable.
\end{proposition}

\begin{proof}
Under $MA_{\omega_1}$, every $L$-space has a finite power which is not Lindel\"of. Now, every $\iota$-space is an $\epsilon$-space.
\end{proof}

To investigate the consistency of the negation of this statement, we focus on more general classes of spaces that yield L-spaces and provide many counterexamples in topology.  These spaces are subspaces of products of the form $2^I$, where $I$ is a set of ordinals.  So, the following notation is useful.

\begin{notation}
\mbox{}
\begin{itemize}
\item $Fn(I,2)$ is the set of finite partial functions from $I$ into $2$.
\item For $\varepsilon \in Fn(I,2)$, $[\varepsilon ]=\{f\in 2^I:\varepsilon \subseteq f\}$ denotes the basic clopen set determined by $\varepsilon $.
\item If $b\in [I]^{<\omega }$ such that $b=\{\beta _i:i\in n=\vert b\vert \}$ and $\varepsilon \in 2^n$ then $\varepsilon *b$ denotes the element of $Fn(I,2)$ which has $b$ as its domain and satisfies $\varepsilon *b(\beta _i)=\varepsilon (i)$, $\forall i\in n$.
\item For any cardinal $\mu $ and $r\in \omega $ we denote by ${\mathcal D}_\mu ^r(I)$ the collection of all sets $B\in [[I]^r]^\mu $ such that the members of $B$ are pairwise disjoint.  We write ${\mathcal D}_\mu (I)=\bigcup \{{\mathcal D}_\mu ^r(I):r\in \omega \}$ and if $B\in {\mathcal D}_\mu (I)$ then $n(B)=\vert b\vert $ for any $b\in B$.
\end{itemize}
\end{notation}

\begin{definition}
If $B\in {\mathcal D}_\mu (I)$ and $\varepsilon \in 2^{n(B)}$ then $[\varepsilon ,B]=\bigcup \{[\varepsilon *b]:b\in B\}$ is called a ${\mathcal D}_\mu $-set in $2^I$.
\end{definition}

\begin{definition}
$X\subseteq 2^\lambda $ with $\vert X\vert >\omega $ is an HFC space if for every $B\in {\mathcal D}_\omega (\lambda )$ and $\varepsilon \in 2^{n(B)}$, $\vert X\setminus [\varepsilon ,B]\vert \leq \omega $.  That is, every ${\mathcal D}_\omega $-set in $2^\lambda $ finally covers $X$. 
\end{definition}

\begin{definition}
For any $k\in \omega $, a map $F:\kappa \times \lambda \rightarrow 2$ with $\kappa \geq \omega _1$ and $\lambda \geq \omega $ $(\lambda \geq \omega _1)$ is called an HFC$^k$ (HFC$_w^k$) matrix if for every $A\in {\mathcal D}_{\omega _1}^k(\kappa )$ and $B\in {\mathcal D}_{\omega }(\lambda )$ $(B\in {\mathcal D}_{\omega _1}(\lambda ))$ and for any $\varepsilon _0,\dots ,\varepsilon _{k-1}\in 2^{n(B)}$ there exists $b\in B$ such that $\vert \{a\in A: \forall i\in k(f_{\alpha _i}\supseteq \varepsilon _i*b)\}\vert =\omega _1$, where $\{\alpha _i:i\in k\}$ is the increasing enumeration of the elements of $a$.

$F$ is a strong HFC (HFC$_w$) matrix  if it is HFC$^k$ (HFC$_w^k$) for all $k\in \omega $.

$X\subseteq 2^\lambda $ is a strong HFC (HFC$_w$) space if it is represented by a strong HFC (HFC$_w$) matrix, $F$.  That is $X=\{f_\alpha :\alpha <\kappa \}$ where $f_\alpha (\gamma )=F(\alpha ,\gamma )$, $\forall \gamma <\lambda $.
\end{definition}
\begin{theorem}\label{strongHFChL}\cite{J}
If $X$ is a strong HFC$_w$ space (hence strong HFC) then $X^k$ is hereditarily Lindel\" of, $\forall k\in \omega $.
\end{theorem}
\begin{corollary}\label{strongHFCiota}
Every strong HFC is an $\iota $-space.  In fact, if $X^n$ is hereditarily Lindel\" of, $\forall n\in \omega $ 
then $X$ is an $\iota $-space.
\end{corollary}

\begin{proof}
Follows from Theorem \ref{powersthm}.
\end{proof}

In contrast
\begin{example}[CH]\label{notiotaHFC}
There is an HFC with no countable open $\iota $-cover.
\end{example}

\begin{proof}
By CH, enumerate the collection of ${\mathcal D}_\omega $-sets in $2^{\omega _1}$ by $\{u_\alpha :\alpha <\omega _1\}$ so that for $n<\omega $, $\{\sigma _{ni}:i\in \omega \}\subseteq Fn(\omega _1,2)$ such that \{dom$(\sigma _{ni}):i\in \omega \}$ is a pairwise disjoint collection of finite subsets of $\omega $ and $u_n=\bigcup _{i\in \omega }[\sigma _{ni}]$.  Moreover, for $\alpha \geq \omega $, $\{\sigma _{\alpha i}:i\in \omega \}\subseteq Fn(\omega _1,2)$ such that \{dom$(\sigma _{\alpha i}):i\in \omega \}$ is pairwise disjoint and $u_\alpha =\bigcup _{i\in \omega }[\sigma _{\alpha i}]$.  For $\alpha \geq \omega $, let ${\mathcal U}_\alpha =\{u_\beta :\beta <\alpha $, dom$(\sigma _{\beta i})\subseteq \alpha $, $\forall i\in \omega \}$.  Enumerate ${\mathcal U}_\alpha =\{v_{\alpha i}:i\in \omega \}$ where each $u\in {\mathcal U}_\alpha $ appears as infinitely many $v_{\alpha i}$'s.  Construct HFCs $X=\{x_\alpha :\omega \leq \alpha <\omega _1\}$ and $Y=\{y_\alpha :\omega \leq \alpha <\omega _1\}$ by induction, defining $x\restriction \alpha $, $y\restriction \alpha $ at stage $\alpha $ and letting $x_\alpha (\gamma )=0$, $\forall \gamma \geq \alpha $, $y_\alpha (\gamma )=1$, $\forall \gamma \geq \alpha $. 

For $\omega \leq \alpha < \omega _1$, define $\{\sigma ^\alpha _i:i\in \omega \}$ such that

\begin{description}
\item[(i)] if $v_{\alpha i}=u_\beta $ then $\sigma ^\alpha _i=\sigma _{\beta j}$ for some $j\in \omega $.
\item[(ii)] \{dom$(\sigma ^\alpha _i):i\in \omega \}$ is pairwise disjoint.
\end{description}
$v_{\alpha 0}\in {\mathcal U}_\alpha \Rightarrow v_{\alpha 0}=u_\beta $ for some $\beta <\alpha $ so let $\sigma^\alpha _0=\sigma _{\beta 0}$.

Fix $n>0$ and suppose $\{\sigma ^\alpha _i:i<n\}$ have been defined.  Again, since $v_{\alpha n}\in {\mathcal U}_\alpha $, let $\gamma <\alpha $ such that $v_{\alpha n}=u_\gamma $, where $u_\gamma =\bigcup \{[\sigma _{\gamma i}]:i\in \omega \}$ with \{dom$(\sigma _{\gamma i}):i\in \omega \}$ pairwise disjoint.  Thus, let $j_n\in \omega $ such that dom$(\sigma _{\gamma j_n})\cap $dom$(\sigma^\alpha _i)=\emptyset $, $\forall i<n$ and $\sigma ^\alpha _n=\sigma _{\gamma j_n}$. 

For $\omega \leq \alpha < \omega _1$, define $x_\alpha $, $y_\alpha \in 2^{\omega _1}$ as follows:
$x_\alpha (\gamma )=y_\alpha (\gamma )=\sigma ^\alpha _i(\gamma )$, $\forall \gamma \in \bigcup _{i\in \omega }$dom$(\sigma ^\alpha _i)$, $x_\alpha (\gamma )=y_\alpha (\gamma )=0$, $\forall \gamma \in \alpha \setminus \bigcup _{i\in \omega }$dom$(\sigma ^\alpha _i)$ and as above, $x_\alpha (\gamma )=0$, $\forall \gamma \geq \alpha $, $y_\alpha (\gamma )=1$, $\forall \gamma \geq \alpha $.

\begin{claim}
$X\cup Y$ is an HFC with no countable open $\iota $-cover.
\end{claim}

To see $X\cup Y$ is an HFC, fix $\beta <\omega _1$ and show $u_\beta $ is a final cover of $X\cup Y$.  Note that $\forall \beta <\omega _1$, $\exists \delta <\omega _1$ such that $u_\beta \in {\mathcal U}_\delta $.

So, let $\delta _\beta =\min \{\delta <\omega _1: u_\beta \in {\mathcal U}_\delta \}$.  Then $X\cup Y \setminus (\{x_\gamma :\gamma <\delta _\beta \}\cup \{y_\gamma :\gamma <\delta _\beta \})\subseteq u_\beta $.

Let ${\mathcal U}=\{U_n:n\in \omega \}$ be any countable open cover of $X\cup Y$.  Since $X\cup Y$ is hereditarily Lindel\" of (being HFC), let $\sigma _n(i)\in Fn(\omega _1,2)$ such that $U_n=\bigcup _{i\in \omega } [\sigma _n(i)]\cap (X\cup Y)$, $\forall n\in \omega $. Let $\alpha _n=\sup (\bigcup _{i\in \omega}$dom$(\sigma _n(i)))<\omega _1$ 
and $\alpha =\sup \{\alpha _n:n\in \omega \}<\omega _1$.  We claim that $\forall \beta >\alpha $, $x_\beta \in U_n \Leftrightarrow y_\beta \in U_n$, $\forall n\in \omega $ and hence ${\mathcal U}$ is not an $\iota $-cover.
Fix $\beta >\alpha $, $n\in \omega $.
\begin{align*}
x_\beta \in U_n &\Leftrightarrow (\exists i\in \omega )x_\beta \in [\sigma _n(i)]\\
&\Leftrightarrow (\exists i\in \omega )\sigma _n(i)\subseteq x_\beta \\
&\Leftrightarrow (\exists i\in \omega )x_\beta (\gamma )=\sigma _n(i)(\gamma ) \forall \gamma \in dom(\sigma _n(i)) \\
&\Leftrightarrow (\exists i\in \omega )y_\beta (\gamma )=\sigma _n(i)(\gamma ) \forall \gamma \in dom(\sigma _n(i)) \\
&\Leftrightarrow (\exists i\in \omega )\sigma _n(i)\subseteq y_\beta \\
&\Leftrightarrow (\exists i\in \omega )y_\beta \in [\sigma _n(i)]\\
&\Leftrightarrow y_\beta \in U_n\\
\end{align*}
\end{proof} 

This gives us another example of an L-space that is not an $\iota $-space, in fact, not even an $\iota _w$-space.  Although we already know consistently (under MA$_{\omega _1}$) that this space is not even an $\epsilon $-space, 
the argument used to show the space has no countable open $\iota $-cover will be used to show what we really want: there is a hereditarily $\epsilon $-space that is not an $\iota _w $-space.  Naively we tried to extend this argument to a strong HFC space (a hereditarily $\epsilon $-space), but along the way we discovered the missing ingredient.  Thus Example \ref{notiotaHFC} also provides an example of a certainly already known result.

\begin{corollary}
There is a pair of strong HFCs whose union is not a strong HFC.
\end{corollary}
\begin{proof}
Let $X=\{x_\alpha :\alpha <\omega _1\}$, $Y=\{y_\alpha :\alpha <\omega _1\}$ be the HFCs from Example \ref{notiotaHFC}.  Let $f_X:X\rightarrow \omega _1$ such that $f_X(x_\alpha )=\alpha $ and $f_Y:Y\rightarrow \omega _1$ such that $f_Y(y_\alpha )=\alpha $.
\begin{claim}
$\exists A\in [\omega _1]^{\omega _1}$ such that $\{x_\alpha :\alpha \in A\}$, $\{y_\alpha :\alpha \in A\}$ are strong HFCs.
\end{claim}
\begin{proof}[Proof of Claim]
Let $\vec{N}=\left\langle N_\alpha :\alpha <\omega _1 \right\rangle $ be an $\omega _1$-chain of countable elementary submodels of some $H_\theta $ such that $X,Y,f_X,f_Y \in N_0$ and $\beta <\alpha <\omega _1 \Rightarrow N_\beta \subsetneq N_\alpha $.  Define by recursion $Z=\{z_\alpha :\alpha <\omega _1\}\subseteq X$, separated by $\vec{N}$:
\\
Let $z_0\in X\cap N_0$
\\
Fix $\alpha >0$ and suppose $\{z_\beta :\beta <\alpha \}$ have been defined such that $z_\beta \in X\cap N_\beta \setminus \bigcup _{\gamma <\beta }N_\gamma $.  Since $X$ is uncountable and $\bigcup _{\beta <\alpha }N_\beta $ is countable, $X\setminus \bigcup _{\beta <\alpha }N_\beta \neq \emptyset $.  So, by elementarity, let $z_\alpha \in X\cap N_\alpha \setminus \bigcup _{\beta <\alpha }N_\beta \neq \emptyset $.
\\
To see $Z$ is separated by $\vec{N}$, let $\{z_\alpha ,z_\beta \}\in [Z]^2$.  Without loss of generality, suppose $\alpha <\beta $.  Then, by construction, $N_\alpha \cap \{z_\alpha ,z_\beta \}=\{z_\alpha \}$ and hence $\exists \alpha <\omega _1$ such that $\vert N_\alpha \cap \{z_\alpha ,z_\beta \}\vert = 1$.
\\
Then, by Theorem 2.1 of \cite{S}, $Z$ is a strong HFC.  Since $Z\in [X]^{\omega _1}$, let $A\in [\omega _1]^{\omega _1}$ such that $Z=\{x_\alpha :\alpha \in A\}$.  We claim that $\{y_\alpha :\alpha \in A\}$ is separated by $\vec{N}$ and hence is a strong HFC (again by Theorem 2.1 of \cite{S}).
\begin{note}
$\forall \alpha ,\gamma <\omega _1$, $x_\alpha \in N_\gamma \Leftrightarrow y_\alpha \in N_\gamma $.
\end{note}
\begin{proof}[Proof of Note]
Suppose $x_\alpha \in N_\gamma $.  Since $f_X\in N_\gamma $, $f_X(x_\alpha )=\alpha \in N_\gamma $ and hence $x_\alpha \restriction \alpha \in N_\gamma $.  Recall that $y_\alpha $ is definable from $x_\alpha \restriction \alpha $, $\alpha \in N_\gamma $ since $y_\alpha \restriction \alpha =x_\alpha \restriction \alpha $ and $y_\alpha (\gamma )=1$, $\forall \gamma \geq \alpha $.  Hence $y_\alpha \in N_\gamma $.  Similarly, $y_\alpha \in N_\gamma \Rightarrow x_\alpha \in N_\gamma $.
\end{proof}
Then it is clear $\{y_\alpha :\alpha \in A\}$ is separated by $\vec{N}$.  If $\{y_\alpha ,y_\beta \}\in [\{y_\alpha :\alpha \in A\}]^2$ then $\{x_\alpha ,x_\beta \}\in [Z]^2$ so $\exists \gamma <\omega _1$ such that $\vert N_\gamma \cap \{x_\alpha ,x_\beta \}\vert =1\Leftrightarrow \vert N_\gamma \cap \{y_\alpha ,y_\beta \}\vert =1$ (by the note).
\end{proof}
Therefore, $\{x_\alpha :\alpha \in A\}, \{y_\alpha :\alpha \in A\}$ are strong HFCs and as in Example \ref{notiotaHFC}, $\{x_\alpha :\alpha \in A\}\cup \{y_\alpha :\alpha \in A\}$ has no countable $\iota $-cover.  Thus, $\{x_\alpha :\alpha \in A\}\cup \{y_\alpha :\alpha \in A\}$ is not an $\iota $-space and hence is not a strong HFC by Corollary \ref{strongHFCiota}.
\end{proof}

Fortunately, considering strong HFC$_w$ spaces and working a little harder provides us with the desired example.  In \cite{J}, Juh\' asz constructs a strong HFC$_w$ space in a generic extension obtained by adding a Cohen or random real (in fact a generic extension with a slightly more general property).  Using this same construction, we obtain two strong HFC$_w$ spaces whose union is a hereditarily $\epsilon $-space but has no countable open $\iota $-cover, hence not $\iota _w$.  This gives an example of a space in which every subspace has any finite power Lindel\" of, but there are two subspaces whose product is not Lindel\" of.  In particular, all squares of subspaces are Lindel\" of, but there is a rectangle that is not Lindel\" of; a hereditarily Lindel\" of space whose square is not hereditarily Lindel\" of.

In comparison to Definition \ref{almostiota},
\begin{definition}\label{almostepsilon}
A space $X$ is almost-$\epsilon $ if for every open $\omega $-cover ${\mathcal U}$ of $X$, there is a countable ${\mathcal V}\subseteq {\mathcal U}$ and $A\in [X]^\omega $ such that ${\mathcal V}$ is an $\omega $-cover of $X\setminus A$.
\end{definition}

\begin{lemma}\label{almostepsilonlem}
If $X$ is almost-$\epsilon $ then $X$ is an $\epsilon $-space.
\end{lemma}
\begin{proof}
Let ${\mathcal U}$ be any open $\omega $-cover of $X$ and ${\mathcal M}$ be a countable elementary submodel of some $H_\theta $ ($\theta $ sufficiently large) such that ${\mathcal U}, (X,\tau )\in {\mathcal M}$.
\begin{claim}
${\mathcal U}\cap {\mathcal M}$ is a countable $\omega $-subcover of ${\mathcal U}$.
\end{claim}
Let $F\in [X]^{<\omega }$ and consider ${\mathcal U}_{F\cap {\mathcal M}}=\{U\in {\mathcal U}:F\cap {\mathcal M}\subseteq U\}\in {\mathcal M}$ (since $F\cap {\mathcal M}\subseteq {\mathcal M}$ is finite).  Notice that ${\mathcal U}_{F\cap {\mathcal M}}$ is an open $\omega $-cover of $X$ and since, by elementarity, ${\mathcal M}\models X$ is almost-$\epsilon $, let ${\mathcal V}\in {\mathcal M}$ be countable and $A\in [X]^\omega \cap {\mathcal M}$ such that ${\mathcal V}\subseteq {\mathcal U}_{F\cap {\mathcal M}}$ is an $\omega $-cover of $X\setminus A$.   Since $A,{\mathcal V}\in {\mathcal M}$ are countable, $A,{\mathcal V}\subseteq {\mathcal M}$.  In particular, ${\mathcal V}\subseteq {\mathcal U}\cap {\mathcal M}$.  Also, since $A\subseteq {\mathcal M}$, $F\setminus {\mathcal M}\in[X\setminus A]^{<\omega }$ so let $V\in {\mathcal V}\subseteq {\mathcal U}_{F\cap {\mathcal M}}$ such that $F\setminus {\mathcal M}\subseteq V$.  Then $V\in {\mathcal U}\cap {\mathcal M}$ such that $F\subseteq V$.
\end{proof}

The following alternate characterization of an HFC$_w^k$ space is an adaptation of the characterization of an HFC$_w$ space from \cite{J}.

\begin{theorem}
For any $k\in \omega $, if $X\subseteq 2^\lambda $ with $\vert X\vert >\omega $ and $\lambda >\omega $ is HFC$_w^k$, then  
\begin{align*}
& \forall B\in {\mathcal D}_{\omega _1}(\lambda ), \tag{$*_k$} \forall \varepsilon _0,\dots ,\varepsilon _{k-1}\in 2^{n(B)} , \exists C\in [B]^\omega , \exists \alpha \in \kappa \\ 
& (\forall a=\{\alpha _i:i<k\}\in[\kappa \setminus \alpha ]^k ) 
(\exists b\in C) f_{\alpha _i}\supseteq \varepsilon _i*b, \forall i\in k.
\end{align*}
\end{theorem}
\begin{proof}
Suppose, by way of contradiction, that there is $B\in {\mathcal D}_{\omega _1}(\lambda )$ and $\varepsilon _0,\dots ,\varepsilon _{k-1}\in 2^{n(B)}$ such that $\forall C\in [B]^\omega $ and $\forall \alpha <\kappa $, $\exists a=\{\alpha _i:i<k\}\in [\kappa \setminus \alpha ]^k$ and $\exists j\in k$ such that $f_{\alpha _j}\nsupseteq \varepsilon _j*b$, $\forall b\in C$.  Enumerate $B=\{b_\gamma :\gamma <\omega _1\}$ and let $C_\mu =\{b_\gamma :\gamma <\mu \}$, $\forall \mu <\omega _1$.  Then $C_\mu \in [B]^\omega $, $\forall \mu <\omega _1$ so define by recursion $\{\alpha _\mu :\mu <\omega _1\}\subseteq \kappa $ so that $A=\{a_\mu :\mu <\omega _1\}\in {\mathcal D}_{\omega _1}^k(\kappa )$ where, by assumption, $a_\mu =\{\alpha _i^\mu :i<k\}\in [\kappa \setminus \alpha _\mu ]^k$ such that $f_{\alpha _j^\mu}\nsupseteq \varepsilon _j*b$, $\forall b\in C_\mu $, for some $j<k$.  Then, since $X$ is HFC$_w^k$, $A\in {\mathcal D}_{\omega _1}^k(\kappa )$, $B\in {\mathcal D}_{\omega _1}(\lambda )$ and $\varepsilon _0,\dots ,\varepsilon _{k-1}\in 2^{n(B)}$, let $b\in B$ such that $\vert \{a\in A:\forall i\in k(f_{\alpha _i}\supseteq \varepsilon _i*b)\}\vert=\omega _1$.  But then $b=b_\mu $ for some $\mu <\omega _1$ and $\{a\in A:\forall i\in k(f_{\alpha _i}\supseteq \varepsilon _i*b)\}\subseteq \{a_\gamma :\gamma \leq \mu \}$ (since $b=b_\mu \in C_\gamma $, $\forall \gamma > \mu $), which is countable and hence we have a contradiction. 
\end{proof}

\begin{theorem}\label{heredepsilonnotiota}
Con(ZFC) $\rightarrow $ Con(ZFC + $ \exists $ hereditarily $\epsilon $-space with no countable open $\iota $-cover).
\end{theorem} 
\begin{proof}
Construct two HFC$_w$ spaces $X=\{x_\alpha :\alpha <\omega _1\}$ and $Y=\{y_\alpha :\alpha <\omega _1\}$, as in (4.2) of \cite{J}, so that $x_\alpha \restriction \alpha =y_\alpha \restriction \alpha =f_\alpha \restriction \alpha $, where $f_\alpha =F(\alpha ,-)=r\circ h_\alpha $ and $\forall \gamma \geq \alpha $, $x_\alpha (\gamma )=0, y_\alpha (\gamma )=1$.
\begin{claim}
$X\cup Y$ is hereditarily-$\epsilon $ but has no countable open $\iota $-cover.
\end{claim}
Let $Z\subseteq X\cup Y$ and ${\mathcal U}$ be any open $\omega $-cover of $Z$.  Without loss of generality, ${\mathcal U}$ consists of finite unions of basic open sets in $2^{\omega _1}$.  That is, $\forall U\in {\mathcal U}$, $U=\bigcup _{i<n_U}[\sigma _i^U]$ with $\sigma _i^U\in Fn(\omega _1,2)$.  Let ${\mathcal M}$ be a countable elementary submodel of $H_\theta $ (for some large enough $\theta $) such that ${\mathcal U}\in {\mathcal M}$.  We claim that ${\mathcal U}\cap {\mathcal M}$ is a countable $\omega $-cover of $Z\setminus (Z\cap {\mathcal M})$ showing that $Z$ is almost-$\epsilon $ and hence an $\epsilon $-space by Lemma \ref{almostepsilonlem}, as required.  To see ${\mathcal U}\cap {\mathcal M}$ is an $\omega $-cover, let $F\in [Z\setminus (Z\cap {\mathcal M})]^{<\omega }$.  Enumerate $F=\{f_{\alpha _0},\dots ,f_{\alpha _{n-1}}\}$ for some $n\in \omega $, where $f_{\alpha _i}=x_{\alpha _i}\in Z\cap X$ or $f_{\alpha _i}=y_{\alpha _i}\in Z\cap Y$ so that if $\exists \beta <\omega _1$ such that $x_\beta ,y_\beta \in Z$, $f_\beta =x_\beta $.  Notice that this enumeration is not a problem since if $F$ is the original set and there is $U\in {\mathcal U}\cap {\mathcal M}$ such that $\{f_{\alpha _0},\dots ,f_{\alpha _{n-1}}\}\subseteq U$, then, as above, since ${\mathcal U}\cap {\mathcal M}$ is countable, enumerate ${\mathcal U}\cap {\mathcal M}=\{U_k:k\in \omega \}$ where $U_k=\bigcup _{i<n_k}[\sigma _i^k]$ and $\gamma _k=\sup (\bigcup _{i<n_k}$dom$(\sigma _i^k))<\omega _1$.   Let $\gamma =\sup \{\gamma _k:k\in \omega \}$.  Then, as above, $\forall \beta >\gamma $, $x_\beta \in U_k \Leftrightarrow y_\beta \in U_k$, $\forall k\in \omega $.  In particular, $x_\beta \in U \Leftrightarrow y_\beta \in U$ $\forall \beta >\gamma $.  Note that $\gamma $ is definable in ${\mathcal M}$ since $U_k\in {\mathcal M}$, $\forall k\in \omega $.   Thus, since $F\notin {\mathcal M}$, $\alpha _i>\gamma $, $\forall i<n$ and hence $F\subseteq U$. 

Since $F\in [Z]^{<\omega }$ and ${\mathcal U}$ is an $\omega $-cover of $Z$, let $U\in {\mathcal U}$ such that $F\subseteq U$.  If $U\in {\mathcal M}$ we are done, so suppose $U\notin {\mathcal M}$.   Since $F\subseteq U=\bigcup _{i<n_U}[\sigma _i^U]$, we can refine $U$ so that $F\subseteq \bigcup_{i<n}[\tau _i]\subseteq U$ with $f_{\alpha _i}\in [\tau _i]$ $\forall i<n$.  We need to define $\tau _0,\dots ,\tau _{n-1}\in Fn(\omega _1,2)$ such that $f_{\alpha _i}\in [\tau _i]\subseteq U$, $\forall i<n$:
\\
Since $f_{\alpha _0}\in [\sigma _{i_0}^U]$ for some $i_0<n_U$, let $\tau _0=\sigma _{i_0}^U$.
\\
Fix $m>0$ and suppose $\{\tau _j:j<m\}$ have been defined.  Since $f_{\alpha _m}\in [\sigma _{i_m}^U]$, if $\tau _j\neq \sigma _{i_m}^U$, $\forall j<m$, then $\tau _m=\sigma _{i_m}^U$.  Otherwise, let $j<m$ such that $\tau _j=\sigma _{i_m}^U$.  Let $N=\{i<m: f_{\alpha _m}\in [\tau _i]\}$ and $\gamma _m=\max \{\bigcup dom(\tau _i):i<m\}<\omega _1$.  Define $dom(\tau _m)=\bigcup _{i\in N}dom(\tau _i)\cup \{\gamma _m+1\}$ and $\tau _m(\alpha )=\tau _i(\alpha)=f_{\alpha _m}(\alpha )$, $\forall \alpha \in dom(\tau _i), i\in N$, $\tau _m(\gamma _m+1)=f_{\alpha _m}(\gamma _m+1)$.
\\
Then $F\subseteq \bigcup _{i<n}[\tau _i]\subseteq U$ and we will further refine $\bigcup _{i<n}[\tau _i]$ so that $F\subseteq \bigcup _{i<n}[\varepsilon _i*b]\subseteq \bigcup _{i<n}[\tau _i]\subseteq U$ for $b\in [\omega _1]^k$ $(k\in \omega )$ and $\varepsilon _0,\dots \varepsilon _{n-1}\in 2^k$.  Let $b=\bigcup _{i<n}dom(\tau _i)\in [\omega _1]^{<\omega }$, $k=\vert b\vert $ and enumerate $b=\{\beta _i:i<k\}$.  For $i<n$, let $\varepsilon _i\in 2^k$ such that $\varepsilon _i(j)=f_{\alpha _i}(\beta _j)$, $\forall j<k$.  Then $\varepsilon _i*b(\beta _j)=\varepsilon _i(j)=f_{\alpha _i}(\beta _j) \Rightarrow \varepsilon _i*b\subseteq f_{\alpha _i}$, $\forall i<n$.   By absoluteness, $\varepsilon _0,\dots ,\varepsilon _{n-1}\in {\mathcal M}$ and if $b\in {\mathcal M}$ then by elementarity, $\exists U\in {\mathcal U}\cap {\mathcal M}$ such that $F\subseteq U$, so suppose $b\notin {\mathcal M}$.  Let $r=b\cap {\mathcal M}$.  Then $b\setminus r\neq \emptyset $.  Let $m=k- \vert r\vert $ and ${\mathcal D}=\{d\in [\omega _1]^m:\bigcup [\varepsilon _i*(r\cup d)]\subseteq V$, for some $V\in {\mathcal U}\}$.  Notice that $b\setminus r\in {\mathcal D}$ 
and since $b\setminus r\notin {\mathcal M}$, ${\mathcal D}$ is uncountable.  Then ${\mathcal D}^\shortmid =\{r\cup d: d\in {\mathcal D}\}$ is an uncountable family of finite subsets of $\omega _1$ so let ${\mathcal B}^\shortmid \subseteq {\mathcal D}^\shortmid $ be an uncountable $\Delta $-system with root $r=b\cap {\mathcal M}$ and let ${\mathcal B}=\{d\setminus r:d\in {\mathcal B}^\shortmid \}\subseteq {\mathcal D}$.   Then ${\mathcal B}\in {\mathcal D}_{\omega _1}(\omega _1)$ such that $\forall b^\shortmid \in {\mathcal B}$, $\vert b^\shortmid \vert =k-\vert r\vert =m$.  Recall $b\setminus r=\{\beta _j: \vert r\vert \leq j<k\}$.  So, reenumerate $b\setminus r=\{\gamma _j:j<m\}$ where $\gamma _j=\beta _{j+\vert r\vert }$, $\forall j<m$.  For $i<n$, let $\varepsilon _i^\shortmid \in 2^m$ such that $\varepsilon _i^\shortmid (j)=\varepsilon _i^\shortmid *b\setminus r(\gamma _j)=\varepsilon _i*b\setminus r (\beta _{j+\vert r\vert })=\varepsilon _i(j+\vert r\vert )$, $\forall j<m$.   Then, since ${\mathcal B}\in {\mathcal D}_{\omega _1}(\omega _1)$, $\varepsilon _0^\shortmid ,\dots ,\varepsilon _{n-1}^\shortmid \in 2^m$ and by elementarity ${\mathcal M}\models (*_k)$, let ${\mathcal C}\in [\mathcal B]^\omega \cap {\mathcal M}$ and $\alpha \in \omega _1\cap {\mathcal M}$ such that $\forall a=\{\beta _i:i<n\}\in [\omega _1\setminus \alpha ]^n$, $\exists c\in {\mathcal C}$ such that $f_{\beta _i}\in [\varepsilon _i^\shortmid *c]$, $\forall i<n$.  In particular, since $F\notin {\mathcal M}$, $\alpha _0,\dots ,\alpha _{n-1}\notin {\mathcal M}$ and since $\alpha \in {\mathcal M}$ we have that $\alpha _i>\alpha $, $\forall i<n$.  Thus, $\{\alpha _i:i<n\}\in [\omega _1\setminus \alpha ]^n$ so let $c\in {\mathcal C}$ such that $f_{\alpha _i}\in [\varepsilon _i^\shortmid *c]$, $\forall i<n$.  But, since ${\mathcal C}\in {\mathcal M}$ is countable, ${\mathcal C}\subseteq {\mathcal M}$ and hence $c\in {\mathcal C}\cap {\mathcal M}$ such that $\varepsilon _i^\shortmid *c\subseteq f_{\alpha _i}$, $\forall i<n$.  Also, since $\varepsilon _i*b\subseteq f_{\alpha _i}$, $\forall i<n$, $f_{\alpha _i}\in [\varepsilon _i\restriction \vert r\vert *r]\cap [\varepsilon _i^\shortmid *c]$, $\forall i<n$.   We claim that $[\varepsilon _i\restriction \vert r\vert *r]\cap [\varepsilon _i^\shortmid *c]=[\varepsilon _i*(r\cup c)]$ and hence $f_{\alpha _i}\in [\varepsilon _i*(r\cup c)]$, $\forall i<n \Rightarrow F\subseteq \bigcup_{i<n}[\varepsilon _i*(r\cup c)]$.
\begin{description}
\item[`$\subseteq $']  Let $g\in [\varepsilon _i\restriction \vert r\vert *r]\cap [\varepsilon _i^\shortmid *c]$.  Then $\varepsilon _i\restriction \vert r\vert *r\subseteq g$ and $\varepsilon _i^\shortmid *c\subseteq g$.   Recall $r=\{\beta _j: j<\vert r\vert \}$ and enumerate $c=\{\gamma _j:j<m\}$.  Then, $g(\beta _j)-\varepsilon _i\restriction \vert r\vert *r(\beta _j)=\varepsilon _i(j)$, $\forall j<\vert r\vert $ and $g(\gamma _j)=\varepsilon _i^\shortmid *c(\gamma _j)=\varepsilon _i^\shortmid (j)=\varepsilon _i(j+\vert r\vert )$, $\forall j<m$.  Now, since $r\cup c=\{\beta _j:j<\vert r\vert \}\cup \{\gamma _j:j<m\}$, reenumerate $r\cup c=\{\alpha _j:j<k\}$ so that $\alpha _j=\beta _j$ for $j<\vert r\vert $ and $\alpha _j =\gamma _{j-\vert r\vert }$ for $\vert r \vert \leq j<k$.  Then, for $j<\vert r\vert $, $g(\alpha _j)=g(\beta _j)=\varepsilon _i(j)=\varepsilon _i*(r\cup c)(\alpha _j)$ and for $j\geq \vert r\vert $, $g(\alpha _j)=g(\gamma _{j-\vert r\vert })=\varepsilon _i((j-\vert r\vert )+\vert r\vert )=\varepsilon _i(j)=\varepsilon _i*(r\cup c)(\alpha _j)$.
\item[`$\supseteq $']  Let $g\in [\varepsilon _i*(r\cup c)]$.  Then $g(\alpha _j)=\varepsilon _i*(r\cup c)(\alpha _j)=\varepsilon _i(j)$, $\forall j<k$, where $\alpha _j $ is defined as above, for $j<k$.  Then, $g(\beta _j)=g(\alpha _j)=\varepsilon _i(j)=\varepsilon _i*(r\cup c)(\beta _j)=\varepsilon \restriction \vert r\vert *r(\beta _j)$, $\forall j<\vert r\vert $ and hence $g\in [\varepsilon _i\restriction \vert r\vert *r]$.  Moreover, $g(\gamma _j)=g(\alpha _{j+\vert r\vert })=\varepsilon _i(j+\vert r\vert )=\varepsilon _i^\shortmid ((j+\vert r\vert )-\vert r\vert )=\varepsilon _i^\shortmid (j)=\varepsilon _i^\shortmid *c(\gamma _j)$, $\forall j<m$ and hence $g\in [\varepsilon _i^\shortmid *c]$.
\end{description}
Since $c\in {\mathcal C}\subseteq {\mathcal D}$, $\bigcup _{i<n}[\varepsilon _i*(r\cup c)]\subseteq V$ for some $V\in {\mathcal U}$.  But $\bigcup _{i<n}[\varepsilon _i*(r\cup c)]\in {\mathcal M}$ (since $r,c, \varepsilon _0, \dots ,\varepsilon _{n-1}\in {\mathcal M}$) so by elementarity, let $V\in {\mathcal U}\cap {\mathcal M}$ such that $\bigcup _{i<n}[\varepsilon _i*(r\cup c)]\subseteq V$.  Then $\exists V\in {\mathcal U}\cap {\mathcal M}$ such that $F\subseteq \bigcup _{i<n}[\varepsilon _i*(r\cup c)]\subseteq V$.
\end{proof}

\section{D-spaces}
The third named author first considered $\iota $-covers when trying to make the $T_2$ hereditarily Lindel\" of non D-space of \cite{SS2} regular.  Since this $T_2$ example is an $\epsilon $-space, he asked in \cite{SS2} whether every regular (hereditarily) $\epsilon $-space is a D-space.  We could ask the same about (hereditarily) $\iota $-spaces.  In fact, it remains unclear whether $\iota $-covers could play a role in constructing such a regular, hereditarily Lindel\" of non D-space.

\begin{corollary}
There is an $\iota_w$-space which is not a $D$-space.
\end{corollary}

\begin{proof}
The example is taken from \cite{AJW}, but we nevertheless present the details of its construction for the reader's convenience. Erik van Douwen showed in \cite{vD} that one can put, on every subset of the real line, a locally compact locally countable topology with countable extent which is finer than the topology it inherits from the Euclidean one. Let $B \subset \mathbb{R}$ be a Bernstein set, that is, a set meeting every uncountable closed set along with its complement. Let $X=\mathbb{R}$, where points of $B$ have neighbourhoods as in the van Douwen topology and points of $X \setminus B$ have their usual Euclidean neighbourhoods.

\vspace{.1in}

\noindent {\bf Claim 1.} $X$ is Lindel\"of and a $D$-space.

\begin{proof}[Proof of Claim 1]
To prove that $X$ is Lindel\"of, let $\mathcal{U}$ be an open cover of $X$. Let $V=\bigcup \{ U \in \mathcal{U}: U \cap (X \setminus B) \neq \emptyset \}$. Then $V$ covers all but countably many points of $X$. Indeed if $X \setminus V$ were uncountable, then $(X \setminus V) \cap (X \setminus B) \neq \emptyset$ which is a contradiction. But since the topology of $X \setminus B$ is Lindel\"of, $\mathcal{V}$ has a countable subcover, and hence $X$ is a Lindel\"of space.

To prove that $X$ is a $D$-space, we need the following lemma:

\begin{lemma}
Every countable space $Y$ is a $D$-space.
\end{lemma}

\begin{proof}[Proof of Lemma]
Let $N: Y \to \tau$ be a neighbourhood assignment and fix an enumeration $\{y_n: n < \omega \}$ of $Y$. Suppose you have picked points $\{y_{k_i}: i \leq n \}$. If $N(\{y_{k_i}: i \leq n \})=Y$ then we are done, otherwise let $y_{k_{n+1}}$ be the least indexed point such that $y_{k_{n+1}} \notin N(\{y_{k_i}: i \leq n \})$. We claim that $N(\{y_{k_i}: i < \omega \})=Y$ and $\{y_{k_i}: i<\omega \}$ is closed discrete. For the first claim, suppose by contradiction that there is a point $y \notin N(\{y_{k_i}: i < \omega\}$. Then there is a $j<\omega$ such that $y$ is the least indexed point (in the original enumeration of $Y$) such that $y \notin \{y_{k_i}: i \leq j \}$. But then $y=y_{k_{j+1}}$, which is a contradiction. The fact that $\{y_{k_i}: i < \omega \}$ is closed discrete follows from the first claim and the fact that $N(y_{k_n}) \cap \{y_{k_i}: i < \omega \}$ is a finite set.
\end{proof}

Let $N: X \to \tau$ be an open neighbourhood assignment. Since $X \setminus B$ is a $D$-space we can find a closed discrete set $D_1 \subset X \setminus B$ such that $N(D_1) \supset X \setminus B$. Now $X \setminus N(D_1)$ is a countable closed set. So $X \setminus N(D_1)$ is a $D$-space and hence we can find $D_2 \subset X \setminus N(D_1)$ such that $N(D_2) \supset X \setminus N(D_1)$. Therefore $D=D_1 \cup D_2$ is a closed discrete set such that $N(D)=X$, so $X$ is a $D$-space.
\end{proof}
Now let $B_e$ be the Bernstein set $B$ with its usual (Euclidean) topology.

\noindent {\bf Claim 2}. $X \times B_e$ is an $\iota_w$-space but not a $D$-space.

\noindent {\it Proof of Claim 2}.  Since the topology on $X$ refines the topology of the real line, and the real line has a countable open $\iota$-cover, also $X$ has a countable open $\iota$-cover. So, it follows from Theorem $\ref{productiw}$ that $X \times B_e$ is an $\iota_w$-space. To prove that $X \times B_e$ is not a $D$-space, note that it contains the closed copy $\{(x,x): x \in B \}$ of the space $B$ and that $B$ is not a $D$-space, because it has countable extent, but it is uncountable and locally countable and thus not Lindel\"of.
\end{proof}

\begin{question}
Is every (hereditarily) $\iota$-space a $D$-space?
\end{question}


\begin{thebibliography}{10}
\bibitem{AJW} O. Alas, L. Junqueira and R. Wilson, \emph{Dually discrete spaces}, Topology Appl. \textbf{155} (2008), 1420--1425.
\bibitem{A} A. V. Arkhangel'ski\v i, \emph{Topological function spaces}. Translated from the Russian by R. A. M. Hoksbergen. Mathematics and its Applications (Soviet Series), \textbf{78}. Kluwer Academic Publishers Group, Dordrecht, 1992.
\bibitem{BG} Z. Balogh and G. Gruenhage, \emph{Base multiplicity in compact and generalized compact spaces} Topology Appl. \textbf{115} (2001), 139--151. 
\bibitem{B} D. K. Burke, \emph{Covering properties}. Handbook of set-theoretic topology, North-Holland, Amsterdam, 1984, pp. 347--422,  
\bibitem{J}I. Juh\' asz, \emph{HFD and HFC type spaces, with applications}, Topology Appl. \textbf{126} (2002), 217--262.
\bibitem{L} L. Lawrence, \emph{Lindel\"of spaces concentrated on Bernstein subsets of the real line}, Proc. Amer. Math. Soc. \textbf{114} (1992), 211--215.
\bibitem{P} T. C. Przymusi\' nski,  \emph{Normality and paracompactness in finite and countable Cartesian products}, Fund. Math. 105 (1979-80), no. 2, 87--104.
\bibitem{SS2} D. Soukup and P. J. Szeptycki, \emph{A counterexample in the theory of D-spaces}, Topology Appl. \textbf{159} (June 2012), no. 10-11, 2669--2678.
\bibitem{S} L. Soukup, \emph{Certain L-spaces under CH}, Topology Appl. \textbf{47} (1992), 1--7.
\bibitem{Tk} V. Tkachuk, \emph{Lindel\"of $\Sigma$-spaces, an omnipresent class}, Revista de la Real Academia de Ciencias Exactas, F'sicas y Naturales. Serie A: Matem\' aticas, \textbf{104} (2010), 221-244. 
\bibitem{vD} E. van Douwen, \emph{A technique for constructing honest locally compact submetrizable examples}, Topology Appl. \textbf{47} (1992), 179--201.
\end{thebibliography}
\end{document}